\documentclass[11pt]{article}
\usepackage{graphicx}
\usepackage{amssymb}
\usepackage{amsmath}
\input{epsf.sty}   
\usepackage{color}
\newtheorem{theorem}{Theorem}[section]

\newtheorem{remark}{Remark}[section]

\newtheorem{proposition}{Proposition}[section]

\newenvironment{proof}
{\begin{trivlist}
\item[\hspace{\labelsep}{\bf\noindent Proof. }]}
{$\hfill\Box$\end{trivlist}}
\begin{document}
\title{Telegraph process with elastic boundary at the origin\footnote{This is a post-peer-review, pre-copyedit version of an article published 
in Methodology and Computing in Applied Probability. The final authenticated version is available online at 
https://doi.org/10.1007/s11009-017-9549-4.}}
\author{Antonio Di Crescenzo\thanks{Dipartimento di Matematica,
Universit\`a degli Studi di Salerno, 84084 Fisciano (SA), Italy. e-mail: \texttt{adicrescenzo@unisa.it}}
\and
Barbara Martinucci\thanks{Dipartimento di Matematica,
Universit\`{a} degli Studi di Salerno, 84084 Fisciano, SA, Italy. e-mail:
\texttt{bmartinucci@unisa.it} (corresponding author)}\and Shelemyahu
Zacks\thanks{Department of Mathematical Sciences, Binghamton University, Binghamton, NY 13902-6000, USA. e-mail:
\texttt{shelly@math.binghamton.edu}.}
}
\date{}
\maketitle

\begin{abstract}
We investigate the one-dimensional telegraph random process in the presence of an elastic boundary 
at the origin. This process describes a finite-velocity random motion that alternates between two 
possible directions of motion (positive or negative). When the particle hits the origin, it is either 
absorbed, with probability $\alpha$, or reflected upwards, with probability $1-\alpha$. 
In the case of exponentially distributed random times between consecutive changes of direction, 
we obtain the distribution of the renewal cycles and of the absorption time at the origin. 
This investigation is performed both in the case of motion starting from the origin and non-zero 
initial state. We also study the probability law of the process within a renewal cycle. 

% 
%\subclass{
%60K15 % Markov renewal processes, semi-Markov processes
%\and
%60J25 % Continuous-time Markov processes on general state spaces
%}
\end{abstract}

%--------------------------------------------------------------------
\section{\bf Introduction}\label{section:1}
%--------------------------------------------------------------------
The (integrated) telegraph process describes an alternating random motion with finite velocity. This stochastic process 
deserves interest in various applied fields, such as physics, finance, and mathematical biology. 
Among the first authors that studied the solution of the telegraph equation we recall Goldstein \cite{Goldstein} and Kac \cite{Kac}. Several aspects and generalization of the  telegraph process have been provided in a quite large literature. 
Orsingher \cite{Orsingher90} studied the probability law, flow function, maximum distribution of wave-governed 
random motions of the telegraph type. The distributions of the first-passage time and of the maximum of the 
telegraph process were obtained by Foong \cite{Foong}.  
The solutions of the one-dimensional telegraph equation on a semi-infinite line terminated by a trap, 
and on a finite line terminated by two traps were determined by Masoliver {\em et al.}\ \cite{Masoliver92}. 
The analysis of the telegraph process in the presence of reflecting and absorbing barriers was also 
investigated in Orsingher \cite{Orsingher95} and Ratanov \cite{Ratanov1997}.
\par
Restricting the attention to some recent contributions, we also mention  Beghin {\em et al.}\ \cite{BeghinNOrsingher} and L\'opez and Ratanov \cite{LoRa2014} for the asymmetric telegraph process, Bogachev and Ratanov \cite{Bogachev} for the distribution of the occupation time
of the positive half-line for the telegraph process, Crimaldi {\em et al.}\ \cite{CDCIMa2013} for a telegraph process driven by certain random trials, De Gregorio and Macci \cite{Macci} for the large deviation principle applied to the telegraph process, Di Crescenzo and Martinucci \cite{DCMa2010} for a damped telegraph process, 
Fontbona {\em et al.}\ \cite{Fontbona} for the long-time behavior of an ergodic variant of the telegraph process, 
Stadje and Zacks \cite{StadjeZacks2004} for the telegraph process with random velocities, 
% Kolesnik \cite{Ko2014}, \cite{Ko2015} for the distributions of combinations of two independent telegraph processes, 
Pogorui {\em et al.}\ \cite{PoRoKo2015} for estimates of the number of  level-crossings for the telegraph process,  
Di Crescenzo and Zacks \cite{DCZa2015} for the analysis of a generalized telegraph process perturbed by 
Brownian motion, De Gregorio and Orsingher \cite{DeGregorioOrsingher} and Garra and Orsingher \cite{GarraOrsingher} 
for certain multidimensional extension of the telegraph process. 
Moreover, D'Ovidio {\em et al.}\ \cite{DOvidio2014} investigate other types of multidimensional 
extensions of the telegraph process, whose distribution is related to space-time fractional 
$n$-dimensional telegraph equations.
A modern treatment of the one-dimensional telegraph stochastic processes, 
with a thorough view to their applications in financial markets, is provided  in 
the book by Kolesnik and Ratanov \cite{KoRa2013}. See also Ratanov \cite{Ratanov2015} for 
a generalization of jump-telegraph process with variable velocities applied to markets modelling. 
\par
Most of the above references are concerning analytical results. However, in some instances one is forced 
to adopt computational methods to solve the governing 
equations. See, for instance, Acebr\'on and Ribeiro \cite{AcebronRibeiro}, where a 
Monte Carlo algorithm is derived to solve the one-dimensional telegraph equations
in a bounded domain subject to suitable boundary conditions.
\par
Several applications of the telegraph process and its numerous generalizations have been stimulated by 
problems involving dynamical systems subject to dichotomous noise. For instance, such processes can 
be used for the description of stochastic dynamics of extended thermodynamic 
theories far from equilibrium (see Giona {\em et al.}\ \cite{Giona}). The need to model physical systems 
in the presence of a variety of complex conditions encouraged several authors to   
analyze stochastic processes restricted by suitable boundaries, such as the elastic ones. 
Examples of papers dealing with elastic boundaries are provided by Veestraeten \cite{Veestraeten} 
and Buonocore {\em et al.}\ \cite{Buonocoreetal}. 
\par
Analytical results on stochastic processes restricted by elastic boundaries have been obtained by 
various authors, such as Domin\'e \cite{Domine}, \cite{Domine2}, for the first-passage problem of the 
Wiener process with drift,  Giorno {\em et al.}\ \cite{Giornoetal} for the construction 
of first-passage-time densities for diffusion processes,  Beghin and Orsingher \cite{BeghinOrsingher} 
for the analysis of fractional diffusion equations. Furthermore, Jacob \cite{Jacob1}, \cite{Jacob2} studied 
a Langevin process with partially elastic boundary at zero and related stochastic differential equations. 
\par
The analysis of finite-velocity random motions subject to elastic boundaries seems to be quite new. 
Along the lines of the previous papers, we investigate the distribution of a one-dimensional 
telegraph process $\{X(t); t\geq 0\}$ in the presence of an elastic boundary at $0$. 
This process describes the motion of a particle over the state space $[0, +\infty)$ and starting at 
$x\geq 0$. The particle moves on the line up and down alternating. For simplicity, 
we assume that the motion has velocity 1 (upward motion) and $-1$ (downward motion). 
Initially, the motion proceeds upward for a positive random time $U_1$. 
After that, the particle moves downward for a positive random time $D_1$, and so on 
the motion alternates along the random times $U_2,D_2,U_3,D_3,\ldots$, where $\{U_i\}_{i\in \mathbb{N}}$ 
and $\{D_i\}_{i\in \mathbb{N}}$ are independent sequences of i.i.d.\ random variables. 
When the particle hits the origin it is either absorbed, with probability $\alpha$ or reflected upwards, 
with probability $1-\alpha$, with $0<\alpha<1$. Specifically, if during a downward period, 
say $D_j$, the particle reaches the origin and is not absorbed, then  instantaneously 
the motion restarts with positive velocity, according to an independent random time $U_{j+1}$. 
\par
The analysis of the telegraph process and related processes is often based on the resolution 
of partial differential equations with proper boundary conditions. However, in this case such 
approach seems to be not fruitful so that we will adopt renewal theory arguments. 
We denote by $C_{x}$ the random time till the first arrival at the origin, with 
starting point  $x\geq 0$, and by 
$C_{0,i}$ the (eventual) $i$th interarrival time between consecutive visits at the origin following 
$C_{x}$, for $i\in \mathbb{N}$. Moreover, let $A_{x}$ denote the time till absorption at the origin 
conditional on initial state $x\geq 0$. Let $M$ be the random number of arriving at the origin, 
until absorption. Clearly, $M$ has a geometric distribution, with 
\begin{equation}
\mathbb{P}(M=m)=\alpha (1-\alpha)^{m-1},\qquad m\in \mathbb{N}, \quad \alpha\in (0,1).
\label{distrM}
\end{equation}
We remark that the random variables $C_{x}, C_{0,1}, C_{0,2}, \ldots$ are independent. 
Moreover, $C_{0,1}, C_{0,2}, \ldots$ are identically distributed, and are called renewal cycles. 
For brevity, we denote by $C_0$  a random variable that is identically distributed as $C_{0,i}$, 
$i\in \mathbb N$. Clearly, the distribution of $C_{x}$ is identical to that of the renewal cycles if $x=0$. 

 \ref{fig:1} shows an example of sample path of $X(t)$, where $D_j^*$ denotes 
the downward random period $D_j$ truncated by the occurrence of the visit at the origin. 
Finally, we point out  the following relation:
\begin{equation}
A_{x}=C_{x}+{\bf 1}_{\{M>1\}}{\sum_{i=1}^{M-1}C_{0,i}}.
\label{RelAx}
\end{equation}
\par
This is the plan of the paper. In Section \ref{section:2} we provide some basic definitions 
and recall some useful results on the distribution of the renewal cycles when $U_i$ are 
exponentially distributed and $D_i$ have a general distribution. 
In Section \ref{section:3} we analyze the absorption time and renewal cycles when the initial 
state is zero, and $U_i$ and $D_i$ have exponential distribution with unequal parameters. 
In this case we obtain the explicit expression of the probability density function (PDF), 
moment generating function (MGF), and moments of $A_{0}$ and $C_{0}$. 
In Section \ref{section:4} we study the absorption time and renewal cycles for non-zero 
initial state. We determine the PDF, the MGF and the moments of $C_{x}$, as well as the 
MGF and the moments of $A_{x}$. Finally, in Section \ref{section:5} we study the 
conditional distribution of $X(t)$ within a renewal cycle. 
\par
The main probabilistic characteristics of the process under investigation will be determined in an 
analytical form. Even if the expressions seem complicated they can be evaluated in standard 
computer environments, as shown in various figures throughout the paper. 
%-------------------------------------------------------------------------------
\begin{figure}[t]\label{fig:1}
\begin{center}
\begin{picture}(341,166) 
\put(20,70){\vector(1,0){350}} 
\put(20,70){\vector(0,1){80}} 
\put(10,150){\makebox(40,15)[t]{$X(t)$}} 
\put(355,48){\makebox(30,15)[t]{$t$}} 
\put(0,57){\makebox(20,15)[t]{0}} 
\put(0,77){\makebox(20,15)[t]{$x$}} 
\put(20,90){\line(1,1){30}} 
\put(50,120){\line(1,-1){40}} 
\put(90,80){\line(1,1){20}} 
\put(110,100){\line(1,-1){30}} 
\put(140,70){\line(1,1){30}} 
\put(170,100){\line(1,-1){20}} 
\put(190,80){\line(1,1){30}} 
\put(220,110){\line(1,-1){40}} 
\put(280,70){\line(1,1){30}} 
\put(310,100){\line(1,-1){30}} 
\put(50,67){\line(0,1){6}} 
\put(90,67){\line(0,1){6}} 
\put(110,67){\line(0,1){6}} 
\put(140,67){\line(0,1){6}} 
\put(170,67){\line(0,1){6}} 
\put(190,67){\line(0,1){6}} 
\put(220,67){\line(0,1){6}} 
\put(260,67){\line(0,1){6}} 
\put(280,67){\line(0,1){6}} 
\put(310,67){\line(0,1){6}} 
\put(340,67){\line(0,1){6}} 
\put(20,44){\line(0,1){8}} 
\put(140,44){\line(0,1){8}} 
\put(260,44){\line(0,1){8}} 
\put(280,44){\line(0,1){8}} 
\put(340,44){\line(0,1){8}} 
\put(70,48){\vector(-1,0){50}} 
\put(70,48){\vector(1,0){70}} 
\put(190,48){\vector(-1,0){50}} 
\put(170,48){\vector(1,0){90}} 
\put(310,48){\vector(-1,0){30}} 
\put(290,48){\vector(1,0){50}} 
\put(20,16){\line(0,1){8}} 
\put(50,16){\line(0,1){8}} 
\put(90,16){\line(0,1){8}} 
\put(110,16){\line(0,1){8}} 
\put(140,16){\line(0,1){8}} 
\put(170,16){\line(0,1){8}} 
\put(190,16){\line(0,1){8}} 
\put(220,16){\line(0,1){8}} 
\put(260,16){\line(0,1){8}} 
\put(280,16){\line(0,1){8}} 
\put(310,16){\line(0,1){8}} 
\put(340,16){\line(0,1){8}} 
\put(264,75){\line(1,0){2}} 
\put(269,75){\line(1,0){2}} 
\put(274,75){\line(1,0){2}} 
\put(264,20){\line(1,0){2}} 
\put(269,20){\line(1,0){2}} 
\put(274,20){\line(1,0){2}} 
\put(264,48){\line(1,0){2}} 
\put(269,48){\line(1,0){2}} 
\put(274,48){\line(1,0){2}} 
\put(50,20){\vector(-1,0){30}} 
\put(30,20){\vector(1,0){20}} 
\put(75,20){\vector(-1,0){25}} 
\put(75,20){\vector(1,0){15}} 
\put(105,20){\vector(-1,0){15}} 
\put(105,20){\vector(1,0){5}} 
\put(125,20){\vector(-1,0){15}} 
\put(125,20){\vector(1,0){15}} 
\put(155,20){\vector(-1,0){15}} 
\put(155,20){\vector(1,0){15}} 
\put(195,20){\vector(-1,0){5}} 
\put(185,20){\vector(1,0){5}} 
\put(210,20){\vector(-1,0){40}} 
\put(205,20){\vector(1,0){15}} 
\put(240,20){\vector(-1,0){20}} 
\put(240,20){\vector(1,0){20}} 
\put(305,20){\vector(-1,0){25}} 
\put(285,20){\vector(1,0){25}} 
\put(335,20){\vector(-1,0){25}} 
\put(315,20){\vector(1,0){25}} 
\put(55,28){\makebox(50,15)[t]{$C_x$}} 
\put(175,28){\makebox(50,15)[t]{$C_{0,1}$}} 
\put(286,28){\makebox(50,15)[t]{$C_{0,m-1}$}} 
\put(10,0){\makebox(50,15)[t]{$U_1$}} 
\put(45,0){\makebox(50,15)[t]{$D_1$}} 
\put(75,0){\makebox(50,15)[t]{$U_2$}} 
\put(100,0){\makebox(50,15)[t]{$D_2^*$}} 
\put(130,0){\makebox(50,15)[t]{$U_3$}} 
\put(155,0){\makebox(50,15)[t]{$D_3$}} 
\put(180,0){\makebox(50,15)[t]{$U_4$}} 
\put(215,0){\makebox(50,15)[t]{$D_4^*$}} 
\put(270,0){\makebox(50,15)[t]{$U_n$}} 
\put(300,0){\makebox(50,15)[t]{$D_n$}} 
\end{picture} 
\end{center}
\vspace{-0.3cm}
\caption{A sample-path of $X(t)$.}
\end{figure}
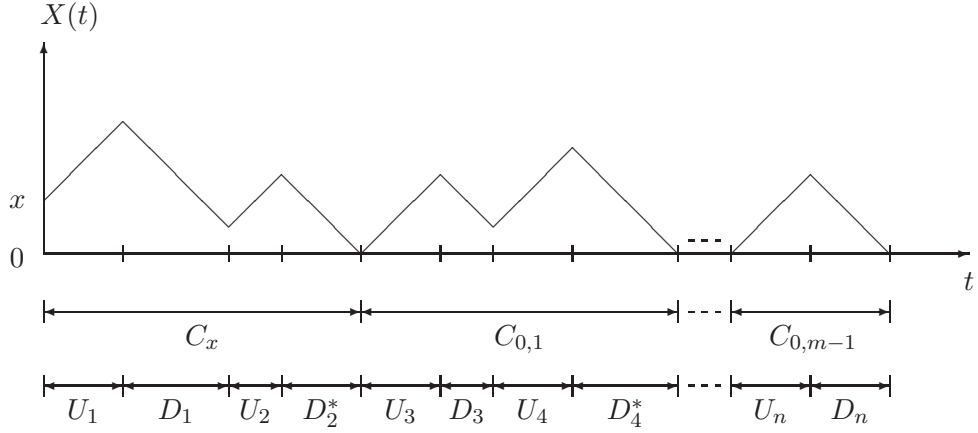 
%-------------------------------------------------------------------------------
%
%----------------------------------------------------------------------------------------
\section{\bf  Preliminaries on the renewal cycles}\label{section:2}
%----------------------------------------------------------------------------------------
Let us denote by $F$ and $G$ the cumulative distribution functions of $U_i$ and $D_i$, 
respectively. We assume that the upward periods of the motion have exponential distribution, i.e. 
\begin{equation}
 F(t)=1-e^{-\lambda t},
 \qquad t\in [0, \infty), \quad  \lambda\in (0, \infty). 
\label{defF}
\end{equation}
Aiming to determine the  distribution of the renewal cycles, we consider the auxiliary compound 
Poisson process
\begin{equation}
Y(t)=\sum_{n=1}^{N(t)}D_{n},
\label{Ydef}
\end{equation}
where 
$$
 N(t)=\max \{n\in \mathbb{N}_0:\sum_{i=1}^{n}U_{i}\leq t\},
$$ 
and thus $N(t)=0$ if $U_{1}>t$. Clearly, $N(t)$ is a Poisson process with intensity $\lambda$, 
so that $\mathbb{P}[Y(t)=0]=e^{-\lambda t}$, $t\in [0, \infty)$, due to (\ref{defF}). 
Moreover, if $Y(t)=s-t$, with $t\in (0,s)$, this means that the total time  (from $0$ to $s$) of 
moving upwards or downwards equals $t$ or $s-t$, respectively. 
The PDF of the absolutely continuous component of $Y(t)$, for $t\in (0, \infty)$, is
\begin{equation}
h(y;t):= \frac{\rm d}{{\rm d}y}\mathbb{P}[Y(t)\leq y]
=e^{-\lambda t}\sum_{n=1}^{+\infty }\frac{(\lambda t)^{n}}{n!}%
g^{(n)}(y),
\qquad y\in (0, \infty),
\label{pdfY}
\end{equation}
where $g^{(n)}(y)$ is the $n$-fold convolution of the PDF of $G$.
\par
Let us now define, for any $x\in [0, \infty)$, the following stopping time: 
\begin{equation}
T_x=\inf \{t>0:Y(t)\geq x+t\}.
\label{Txdef}
\end{equation}
If the motion starts from the origin, i.e.\ $x=0$, then all renewal cycles $C_{0,i}$, $i\in \mathbb{N}$, 
are distributed as $C_{0}$. 
In this case, since the first visit to the origin occurs at the first instant in which the total time downward 
is greater or equal to that of the time upward, we have 
\begin{equation}
C_{0}\stackrel{d}{=} 2T_{0}, \qquad i\in \mathbb{N},
\label{relTC0}
\end{equation}
where $\stackrel{d}{=}$ means equality in distribution. 
When the initial state is away from the origin, i.e.\ $x\in (0, \infty)$, similarly it is   
\begin{equation}
C_{x}\stackrel{d}{=}  x+2T_{x}.
\label{relTCx}
\end{equation}
\par
Notice that the stopping time (\ref{Txdef}) is not necessarily a proper random variable. Indeed, 
$\mathbb{P}(T_{x}<\infty )=1$ for $x\in [0, \infty)$ if and only if $\mathbb E[D_1]\geq \mathbb E[U_1]$, 
and the moments of $T_{x}$ are finite only if $\mathbb E[D_1]> \mathbb E[U_1]$ (see, for instance, 
section 3 of Zacks {\em et al.}\ \cite{Zacks1999}). 
\par
For all $x\in [0, \infty)$, let us now introduce the following subdensity,
\begin{equation}
 g_x(y,t):=\frac{\rm d}{{\rm d}y}\mathbb{P}[Y(t)\leq y, T_x>t], \qquad y\in (0, \infty),\ t\in (0, \infty),
 \label{eq:gxyt}
\end{equation}
and the PDF of the stopping time $T_x$,
\begin{equation}
\psi _{x}(t):=\frac{\rm d}{{\rm d}t}\mathbb{P}(T_x\leq t), \qquad   t\in (0, \infty).
 \label{eq:defpsi}
\end{equation}
The following proposition recalls some useful results obtained by Stadje and Zacks \cite{StadjeZacks03}, 
concerning the functions introduced in (\ref{eq:gxyt}) and (\ref{eq:defpsi}). 
\begin{proposition}
(i) \ The subdensity given in (\ref{eq:gxyt}) can be expressed in terms of $h(y,t)$, defined in Eq.\ (\ref{pdfY}), 
as follows. \\
-- \ If $x=0$, then 
\begin{equation}
 g_0(y,t)=\frac{t-y}{t} h(y,t),\qquad  t\in (0, \infty), \ y\in (0, t).
 \label{g0def}
\end{equation}
-- \ If $x\in (0, \infty)$, then, for $0<y<x+t$ and  $t\in (0, \infty)$,
\begin{eqnarray}
&& \hspace*{-1.5cm}
g_{x}(y,t)={\bf 1}_{\{0<y\leq x\}} h(y;t)+{\bf 1}_{\{x<y<x+t\}}\bigg[h(y;t)-h(y;y-x)e^{-\lambda (t-y+x)}
\nonumber
\\
&& \hspace*{-0.3cm}
-(t-y+x)\int_{t+x-y}^{t}\frac{1}{u} h(u-t+y-x;u)h(t-u+x;t-u)du\bigg].
\label{gxdef}
\end{eqnarray}
(ii) \ For all $x\in [0, \infty)$, the PDF of the stopping time $T_x$ is given by
\begin{equation}
\psi _{x}(t)=\lambda e^{-\lambda t}
 \overline{G}(t+x)+\lambda \int_{0}^{t+x}g_{x}(y,t)\overline{G}(t-y+x)dy,\qquad t\in (0, \infty),
\label{psix}
\end{equation}
where $\overline{G}(t)=1-G(t)=\mathbb{P}(D_1> t)$.
\end{proposition}
\par
In the sequel we assume that the distribution of the downward random times $D_n$ is exponential 
with parameter $\mu$, i.e.
\begin{equation}
 G(t)=1-e^{-\mu t},\qquad  t\in [0, \infty), 
 \label{defG}
\end{equation}
with $0<\mu<\lambda$ in order to ensure that $\mathbb E[D_1]> \mathbb E[U_1]$.  
%--------------------------------------------------------------------
\section{\bf  Absorption time and renewal cycles for zero initial state}\label{section:3}
%--------------------------------------------------------------------
In the present section we consider the special case of initial state $x=0$. 
Recalling that $U_1$ and $D_1$ are exponentially distributed with parameters 
$\lambda$ and $\mu$, respectively, with $0<\mu<\lambda$, from Eq.\ (\ref{pdfY}) we have 
\begin{equation}
h(y,t)= \frac{\sqrt{\lambda \mu t}}{\sqrt{y}} I_1\left(2 \sqrt{\lambda \mu t y}\right)\ e^{-\lambda t-\mu y},
\qquad y\in (0, \infty), \ \ t\in (0, \infty),
\label{hexp}
\end{equation}
where 
\begin{equation}
I_n(z)=\left(\frac{z}{2}\right)^n\, \sum_{k=0}^{+\infty} \frac{\left(\frac{1}{4} z^2 \right)^k}{k! (k+n)!}, 
\qquad n\in {\mathbb N}
\label{modBessel}
\end{equation}
is the $n$-th modified Bessel function of the first kind. We are now able to 
obtain the PDF of (\ref{Txdef}) when $x=0$. 
\begin{proposition}\label{prop:psi0}
Under assumptions (\ref{defF}) and (\ref{defG}), with $0<\mu<\lambda$, the PDF of $T_0$ is given by
\begin{equation}
\psi_0(t)=\frac{\lambda e^{-(\lambda + \mu) t} }{t \sqrt{\lambda \mu}}I_1\left(2 t \sqrt{\lambda \mu}\right), 
\qquad t\in (0, \infty).
\label{psi0exp}
\end{equation}
\end{proposition}
\begin{proof}
From Eq.\ (\ref{psix}), and recalling Eqs.\ (\ref{defF}) and (\ref{defG}), it follows that, for $t\in (0, \infty)$, 
$$
\psi_0(t)=\lambda e^{-(\lambda + \mu) t} \left[1+\sqrt{\frac{\lambda \mu}{t}} \int_{0}^{t}
\frac{(t-y)}{\sqrt{y}} I_1\left(2 \sqrt{\lambda \mu t y}\right) \,{\rm d}y\right].
$$
Hence, after a change of variable, and recalling Eq.\ $(1.11.1.1)$ of Prudnikov {\em et al.}\ \cite{Prudnikov2}, 
we get
\begin{equation}
\psi_0(t)= \lambda e^{-(\lambda + \mu) t} \left\{1+ \frac{ {\lambda \mu t^2}}{2} 
 \left[ 2\; {}_{1}F_{2}(1;2,2;\lambda \mu t^2)-{{}_{1}F_{2}(2;3,2;\lambda \mu t^2)}\right]\right\},
\label{psi0hyper}
\end{equation}
where
\begin{equation}
 {}_{1}F_{2}(a;b,c;z)=\sum_{n=0}^{+\infty}
 \frac{(a)_n}{(b)_n(c)_n}\,\frac{z^n}{n!} 
 \label{Hyper1F2}
\end{equation}
is the hypergeometric function, with $(d)_0=1$ and $(d)_n=d(d+1)\cdots(d+n-1)$ for 
$n\in \mathbb{N}$ (the rising factorial). 
Making use of identities $(1)$ of \cite{Wolfram1} and $(5)$ of \cite{Wolfram2}, from Eq.\ (\ref{psi0hyper}) we obtain 
$$
\psi_0(t)= \lambda e^{-(\lambda + \mu) t} \left[1+ \frac{ {\lambda \mu t^2}}{2}    
\left\{ \frac{2}{(\lambda \mu t^2)^{3/2}}  I_1\left(2 t \sqrt{\lambda \mu }\right)-\frac{2}{\lambda \mu t^2}\right]\right\},
$$
so that Eq.\ (\ref{psi0exp}) finally follows.
\end{proof}
\par
We remark that the PDF given in (\ref{psi0exp}) identifies with the busy period PDF of an M/M/1 queue with 
arrival rate $\mu$ and service rate $\lambda$. 
\par
Let us now study the renewal cycle for zero initial state. We recall that the Gauss hypergeometric function 
is defined as 
\begin{equation}
 {}_{2}F_{1}(a,b;c;z)=\sum_{n=0}^{+\infty}
 \frac{(a)_n (b)_n}{(c)_n}\,\frac{z^n}{n!}.
 \label{Hyper2F1}
\end{equation}
\begin{proposition}\label{propdensC0}
Under the assumptions of Proposition \ref{prop:psi0}, the PDF of $C_0$ is given by
\begin{equation}
f_{C_0}(y)= \frac{1}{y}\sqrt{\frac{\lambda}{\mu}}{\lambda e^{-\frac{(\lambda + \mu)}{2} y}} I_1\left(y \sqrt{\lambda \mu}\right), 
\qquad y\in (0, \infty).
\label{fc0}
\end{equation}
The $n$th moment of $C_0$ is 
\begin{equation}
\mathbb E(C_0^n)= \frac{\lambda \, 2^n n!} {(\lambda+\mu)^{n+1} }\,
 {}_{2}F_{1}\left(\frac{n+1}{2}, \frac{n+2}{2}; 2; \frac{4 \lambda \mu}{(\lambda+\mu)^2}\right),
 \qquad n\in \mathbb{N},
\label{momC0}
\end{equation}
with mean and variance 
$$
\mathbb E(C_0)=\frac{2}{\lambda-\mu},
\qquad 
Var(C_0)=\frac{4(\lambda+\mu)}{(\lambda-\mu)^3}.
$$
\end{proposition}
\begin{proof}
From assumptions (\ref{defF}) and (\ref{defG}), due to relation (\ref{relTC0}), we immediately obtain the 
PDF (\ref{fc0}). Hence, the moments (\ref{momC0}) follow recalling Eq.\ $(3.15.1.2)$ of \cite{Prudnikov4}.
\end{proof}
\par
In the following proposition we obtain the expression of the MGF of the absorption time $A_0$. 
\begin{proposition}
Under the same assumptions of Proposition \ref{prop:psi0}, for $s<(\sqrt{\lambda}-\sqrt{\mu})^2/{2}$, 
the MGF of $A_0$ is
\begin{equation}
M_{A_0}(s):=\mathbb E({\rm e}^{s A_0})
 =\frac{2 \alpha \lambda}{2 \lambda (\alpha-1)+(\lambda+\mu-2 s)+\sqrt{(\lambda+\mu-2 s)^2 -4 \lambda \mu}}.
\label{fgmA0}
\end{equation}
\end{proposition}
\begin{proof}
From Eq.\ (\ref{RelAx}), and recalling Eq.\ (\ref{distrM}), we have  
\begin{equation}
M_{A_0}(s)= \sum_{m=1}^{+\infty} \left[ M_{C_0}(s) \right]^m \mathbb P(M=m) 
 =\frac{\alpha M_{C_0}(s)}{1+(\alpha-1) M_{C_0}(s)},
\label{fgmA0:bis}
\end{equation}
where $M_{C_0}(s)$ is the MGF of $C_0$. 
Due to Eq.\ (\ref{fc0}), and recalling Eq.\ $(3.15.1.8)$ of \cite{Prudnikov4}, 
for $s<(\sqrt{\lambda}-\sqrt{\mu})^2/{2}$ it is 
\begin{eqnarray}
 M_{C_0}(s) &=& \sqrt{\frac{\lambda}{\mu}} \int_{0}^{+\infty} \frac{{\rm e}^{-((\lambda+\mu)/2-s) y}}{y} 
 I_{1}\left(y \sqrt{\lambda \mu}\right) {\rm d}y
 \nonumber \\
 &=& \frac{ (\lambda+\mu-2 s)-\sqrt{(\lambda+\mu-2 s)^2 -4 \lambda \mu}}{2 \mu}.
\label{fgmC0}
\end{eqnarray}
Finally, Eq.\ (\ref{fgmA0}) immediately follows from (\ref{fgmA0:bis}) and (\ref{fgmC0}). 
\end{proof}
\par
In the following theorem we obtain the PDF of the absorption time $A_0$.
\begin{theorem}
Under the same assumptions of Proposition \ref{prop:psi0}, for $y\in (0, \infty)$ we have
\begin{eqnarray}
f_{A_0}(y) & =& 
\alpha \frac{ e^{-\frac{(\lambda + \mu)}{2} y}}{y}  \Bigg\{ \sqrt{\frac{\lambda}{\mu}}I_1\left(y \sqrt{\lambda \mu}\right)
+ \sum_{m=2}^{+\infty} \frac{(\lambda y/2)^m (1-\alpha)^{m-1}}{(m-1)(m-1)!}  
\nonumber
\\
&\times &   \left[2 m \;{}_{1}F_{2}\left(\frac{m-1}{2};\frac{m+1}{2},m;\frac{\lambda \mu y^2}{4}\right)\right.
\nonumber
\\
& -& \left.(m+1) \;{}_{1}F_{2}\left(\frac{m-1}{2};\frac{m+1}{2},m+1;\frac{\lambda \mu y^2}{4}\right)
\right]\Bigg\},
\label{pdfA0}
\end{eqnarray}
with ${}_{1}F_{2}$ defined in (\ref{Hyper1F2}). 
\end{theorem}
\begin{proof}
Denoting by
\begin{equation}
 {\cal L}_s[f(t)]=\int_{0}^{+\infty} {\rm e}^{-s t} f(t) {\rm d}t,\qquad s\in [0, \infty),
 \label{eq:defLsft}
\end{equation}
the Laplace transform of a 
certain  integrable function $f(t)$, from (\ref{fgmA0:bis}) we have
\begin{equation}
 {\cal L}_s[f_{A_0}(t)]= \sum_{m=1}^{+\infty} \left[ \frac{ (\lambda+\mu+2 s)-\sqrt{(\lambda+\mu+2 s)^2 -4 \lambda \mu}}{2 \mu} \right]^m \mathbb P(M=m).
 \label{LTpdfA0}
\end{equation}
We recall that, due to Eqs.\ (2.1.9.18) and (1.1.1.8) of \cite{Prudnikov5}, it is, for $a=2\sqrt{\lambda \mu}$,
\begin{eqnarray}
 & &  {\cal L}_s \left[\int_{0}^{t} \left\{\frac{m}{x} (2 \sqrt{\lambda \mu})^{m+1} 
 I_{m-1} (2 x \sqrt{\lambda \mu}) \right.\right.
 \nonumber \\
 & & \quad \left. \left. -\frac{m(m+1)}{x^2} (2 \sqrt{\lambda \mu})^{m} I_{m} (2 x \sqrt{\lambda \mu})
 \right\}{\rm d}x\right]
 =\left(s-\sqrt{s^2-a^2}\right)^m,
 \label{intLT1}
\end{eqnarray}
for $m\in \mathbb{N}$, $m\geq 2$, where $I_n(\cdot)$ is defined in (\ref{modBessel}). 
Moreover, from Eq.\ (1.11.1.1) of \cite{Prudnikov2}, we have 
\begin{eqnarray}
&& 
\int_{0}^{t} \left\{\frac{m}{x} (2 \sqrt{\lambda \mu})^{m+1} 
I_{m-1} (2 x \sqrt{\lambda \mu})-\frac{m(m+1)}{x^2} (2 \sqrt{\lambda \mu})^{m} I_{m} (2 x \sqrt{\lambda \mu})\right\}{\rm d}x
\nonumber \\
 &&  \hspace*{1.5cm} 
 =\frac{(2 \lambda \mu)^m t^{m-1}}{(m-1) (m-1)!}
 \left\{ 2m \;{}_{1}F_{2}\left(\frac{m-1}{2};\frac{m+1}{2},m;{\lambda \mu t^2}\right)\right.
 \nonumber
 \\
 && \hspace*{1.5cm} 
 \left. -(m+1) \;{}_{1}F_{2}\left(\frac{m-1}{2};\frac{m+1}{2},m+1;{\lambda \mu t^2}\right)
\right\}.
\label{intLT2}
\end{eqnarray}
Hence, from Eq.\ (\ref{LTpdfA0}), taking the inverse Laplace transformation, 
due to Eqs.\ (\ref{fgmC0}), (\ref{intLT1}) and (\ref{intLT2}), and recalling Eq.\ (1.1.1.4) of \cite{Prudnikov5}  
the proof finally follows.
\end{proof}
\par
In Figure 2 we provide some plots of the PDF $f_{A_0}(y)$ for various choices of $\alpha$. 
Such density is decreasing in $y$, with $f_{A_0}(0)=\alpha\lambda/2$. 
%
%--------------------------------------------------------------------
\begin{figure}[t]  %%%%%%%%%%% FIGURA  2
%\vspace{20mm}
\begin{center}
\centerline{
\epsfxsize=6.5cm
\epsfbox{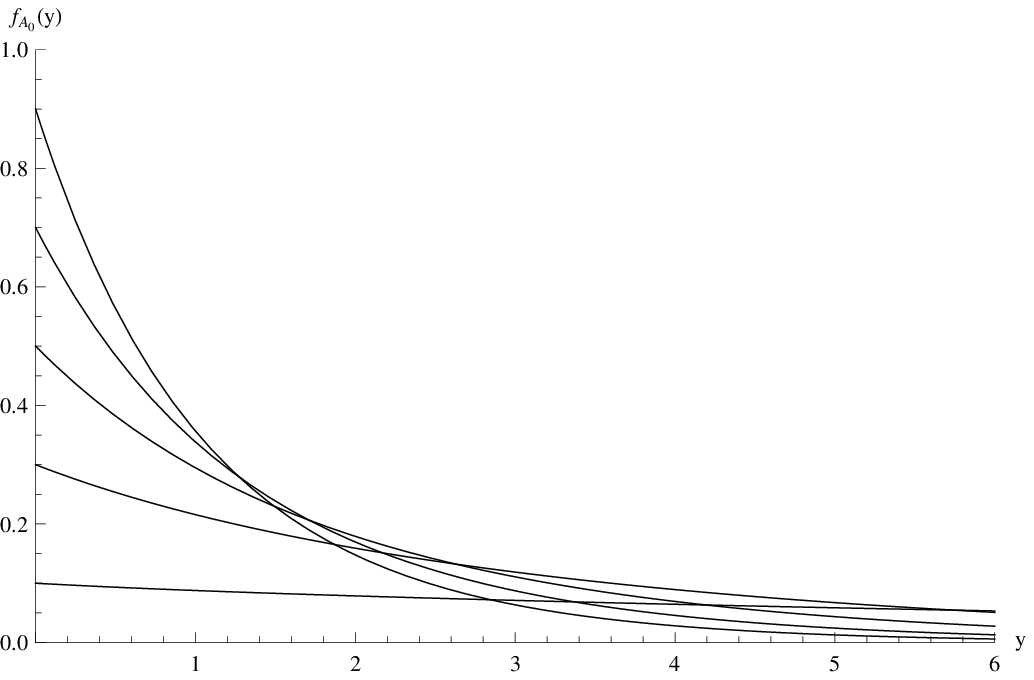}
\hspace*{0.2cm}
\epsfxsize=6.5cm
\epsfbox{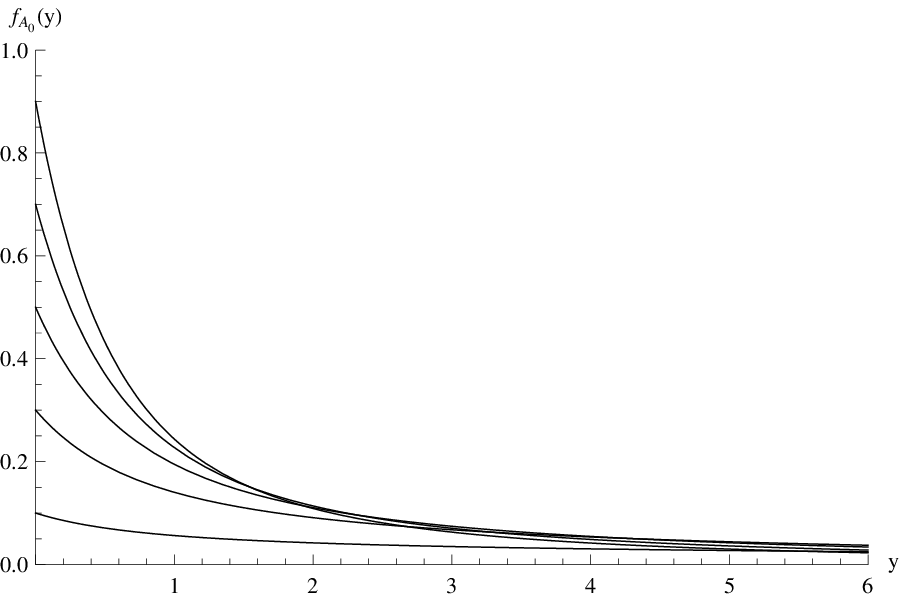}
}
\end{center}
\vspace{-0.5cm}
\caption{Density $f_{A_0}(y)$, given in (\ref{pdfA0}),
for $(\lambda,\mu)=(2,0.5)$ (left-hand side) and  $(\lambda,\mu)=(2,1.5)$ (right-hand side) 
with $\alpha=0.1,0.3,0.5,0.7,0.9$ from bottom to top near the origin.}
\end{figure}
%--------------------------------------------------------------------  
%
\par
We conclude this section by evaluating the moments of the absorption time $A_0$. 
\begin{proposition}
Under the same assumptions of Proposition \ref{prop:psi0}, 
for $n\in \mathbb{N}$ the $n$th moment of $A_0$ is given by
\begin{eqnarray}
\mathbb E(A_0^n) &=&  \frac{2 \alpha \lambda \, n!} {[4 \lambda \alpha ( \mu+ \lambda (\alpha-1))]^{n+1} }
\Bigg\{[2 \mu + 2 \lambda (\alpha-1)][8 \lambda (\alpha-1)]^n
\nonumber
\\
&+& \sum_{h=1}^n [4 \lambda \alpha ( \mu+ \lambda (\alpha -1))]^{h} \,
  [8 \lambda (\alpha-1)]^{n-h}  \frac{\lambda \mu \, 2^{h+1}} {(\lambda+\mu)^{h+1} }
\nonumber
\\  
&\times &
 {}_{2}F_{1}\left(\frac{h+1}{2}, \frac{h+2}{2}; 2; \frac{4 \lambda \mu}{(\lambda+\mu)^2}\right)
\Bigg\}.
\label{momA0}
\end{eqnarray}
The mean  and the variance of $A_0$ are given by
$$
\mathbb E(A_0)=\frac{2}{\alpha(\lambda-\mu)},\qquad 
Var(A_0)=\frac{4[\lambda+\mu(2 \alpha-1)]}{\alpha^2 (\lambda-\mu)^3}.
$$
\end{proposition}
\begin{proof}
From Eqs.\ (\ref{fgmA0}), (\ref{fgmC0}) and  (\ref{momC0}) we have that the 
MGF of $A_0$, for $s<(\sqrt{\lambda}-\sqrt{\mu})^2/{2}$ can be rewritten as
\begin{eqnarray*}
&&  \hspace{-0.2cm}
M_{A_0}(s)=\frac{2 \alpha \lambda\left[2 \lambda (\alpha-1)+(\lambda+\mu-2 s)-\sqrt{(\lambda+\mu-2 s)^2 -4 \lambda \mu}   \right]}{4 \lambda \alpha ( \mu+ \lambda (\alpha-1))-8 \lambda (\alpha-1) s}
\\
&&  
=\frac{2 \alpha \lambda\left[2 \lambda (\alpha-1)+2 \mu +\displaystyle{\sum_{r=1}^{+\infty} \frac{2^{r+1} \lambda \mu}{(\lambda+\mu)^{r+1}}\;{}_{2}F_{1}\left(\frac{r+1}{2}, \frac{r+2}{2}; 2; \frac{4 \lambda \mu}{(\lambda+\mu)^2}\right) s^r}
 \right]}{4 \lambda \alpha ( \mu+ \lambda (\alpha-1))-8 \lambda (\alpha-1) s}.
\end{eqnarray*}
Hence, the proof follows after straightforward calculations.
\end{proof}
%
%--------------------------------------------------------------------
\section{\bf  Absorption time and renewal cycles for non-zero initial state}\label{section:4}
%--------------------------------------------------------------------
%
In this section we obtain the distribution of the renewal cycles in the case $x\in (0, \infty)$. 
We first determine the PDF of the first-passage-time (\ref{Txdef}) for non-zero initial state. 
\begin{proposition}
Under assumptions (\ref{defF}) and (\ref{defG}), for $0<\mu <\lambda$, the PDF of $T_x$, $x>0$, 
for $t\in (0, \infty)$ is given by
\begin{eqnarray}
&& \hspace*{-0.2cm}
\psi_x(t)= \lambda e^{-(\lambda + \mu) t}  e^{-\mu x} \left\{ I_0(2 \sqrt{\lambda \mu t (t+x)})
+\frac{1}{2} \sum_{r=0}^{+\infty} \frac{(\lambda \mu t x)^r}{r! (r+1)!}
\sum_{j=0}^r {r \choose j}  \right.
\nonumber
\\
&& \hspace*{0.2cm}
\left. 
\times (j+r+1) \left(\frac{t}{x}\right)^j 
\left[-1+  {}_{1}F_{2}\left(-\frac{1}{2};\frac{(j+r+1)}{2},1+\frac{(j+r)}{2};\lambda \mu t^2\right)\right]\right\},
\qquad \ 
\label{psixexp_pos}
\end{eqnarray}
where $I_0(\cdot)$ and ${}_{1}F_{2}(a;b,c;\cdot)$ are defined in Eqs.\ (\ref{modBessel}) 
and (\ref{Hyper1F2}), respectively. 
\label{propdensTxpos}
\end{proposition}
\begin{proof}
Substituting (\ref{hexp}) in Eq.\ (\ref{gxdef}), considering the series form of $I_1$, and making use 
of Eq. (2.2.6.1) of \cite{Prudnikov1}, we have
\begin{eqnarray*} 
 g_x(y,t) &=& {\bf 1}_{\{0<y\leq x\}} {\rm e}^{-\lambda t -\mu y} 
 \frac{\sqrt{\lambda \mu t}}{\sqrt{y}} I_1(2 \sqrt{\lambda \mu t y})
 +{\bf 1}_{\{x<y<x+t\}}{\rm e}^{-\lambda t -\mu y} 
 \\
 & \times &  
 \Bigg\{ \frac{\sqrt{\lambda \mu t}}{\sqrt{y}} I_1(2 \sqrt{\lambda \mu t y})- \frac{\sqrt{\lambda \mu (y-x)}}{\sqrt{y}} 
 I_1(2 \sqrt{\lambda \mu (y-x) y}) 
 \\
 &- &  \sum_{k=0}^{+\infty} \sum_{r=0}^{+\infty} \sum_{j=0}^k  \frac{ y^r}{r!  } (j+k)!  
  \frac{(\lambda \mu)^{k+r+2} (t+x-y)^{k+1-j} (y-x)^{k+r+j+2}}
 {(k+r+j+2)! (k+1)! j! (k-j)!}  
 \\
 &\times &  
  \, {}_{2}F_{1}\left(j+k+1,-r; j+k+r+3; \frac{y-x}{y} \right)\Bigg\},
%\label{gxdue}
\end{eqnarray*}
with ${}_{2}F_{1}(a,b;c;\cdot)$ defined in Eq.\ (\ref{Hyper2F1}). Hence, using the above expression of 
$g_x(y,t)$ in Eq.\ (\ref{psix}), and recalling that $\overline G(t)={\rm e}^{-\mu t}$, $t\in [0, \infty)$, for 
$0< \mu< \lambda$ (due to (\ref{defG})), we obtain 
\begin{eqnarray}
 \psi_x(t) &= & \lambda {\rm e}^{-(\lambda+\mu) t -\mu x}
 \Bigg\{ 1+\sqrt{\lambda \mu t} \int_{0}^{t+x} \frac{I_1(2 \sqrt{\lambda \mu t y})}{\sqrt{y}} {\rm d}y
 \nonumber \\
 &- & \int_{x}^{t+x} \frac{\sqrt{\lambda \mu (y-x)}  I_1(2 \sqrt{\lambda \mu (y-x)y})}{\sqrt{y}} {\rm d}y
 \nonumber \\ 
 & - &   \sum_{k=0}^{+\infty} \sum_{r=0}^{+\infty} \sum_{j=0}^k \frac{(\lambda \mu)^{k+r+2} (j+k)! }
 {(k+r+j+2)! (k+1)! j! r!(k-j)!}  
 \nonumber  \\
  &\times &   \int_{x}^{t+x} y^r (t+x-y)^{k+1-j} (y-x)^{k+r+j+2}
 \nonumber  \\
  &\times &  \, {}_{2}F_{1}\left(j+k+1,-r; j+k+r+3; \frac{y-x}{y}\right) {\rm d}y \Bigg\},
\label{psixuno}
\end{eqnarray}
for $t\in (0, \infty)$. 
Due to Eq.\ (1.11.1.1) of \cite{Prudnikov2} and  Eq.\ (7.14.2.84) of \cite{Prudnikov3}, we have
\begin{equation}
\int_{0}^{t+x} \frac{I_1(2 \sqrt{\lambda \mu t y})}{\sqrt{y}} {\rm d}y=\frac{1}{\sqrt{\lambda \mu t}} \left[I_0(2\sqrt{\lambda \mu t (t+x)}) -1 \right],
\label{intpsix1}
\end{equation}
whereas, from  Eq.\ (2.2.6.1) of \cite{Prudnikov1}, it is 
\begin{eqnarray}
&&\hspace*{-0.5cm}
\int_{x}^{t+x} \frac{\sqrt{\lambda \mu y (y-x)} I_1(2 \sqrt{\lambda \mu (y-x) y})}{\sqrt{y}} {\rm d}y=
\sum_{k=0}^{+\infty} \frac{(\lambda \mu)^{k+1}}{k!(k+1)!} \int_{0}^t (x+z)^k z^{k+1} {\rm d}z
\nonumber
 \\
&&\hspace*{1.8cm}
=\lambda \mu t^2 \sum_{k=0}^{+\infty} \frac{(\lambda \mu t x)^{k}}{k!(k+2)!}  \, {}_{2}F_{1}\left(-k,k+2; k+3; -\frac{t}{x}\right).
\label{intpsix2}
\end{eqnarray}
Moreover, recalling Eq.\ (2.21.1.4) of \cite{Prudnikov3}, we get
\begin{eqnarray}
&&\hspace*{-0.2cm}
\int_{x}^{t+x} y^r (t+x-y)^{k+1-j} (y-x)^{k+r+j+2} 
\nonumber \\
&&\hspace*{0.2cm}
 \times \,  {}_{2}F_{1}\left(j+k+1,-r; j+k+r+3; \frac{y-x}{y}\right) {\rm d}y
\nonumber
 \\
&&\hspace*{0.2cm}
=\frac{t^{2 k+r+4} x^r  (k+1-j)! (k+r+j+2)!}{(2k+r+4)!} \,{}_{2}F_{1}\left(-r,r+2; 2k+r+5;-\frac{t}{x} \right). 
\nonumber 
\\
\label{intpsix3}
\end{eqnarray}
Hence, making use of Eqs.\ (\ref{intpsix1}), (\ref{intpsix2}) and (\ref{intpsix3}) in 
Eq.\ (\ref{psixuno}), for $t\in (0, \infty)$ we obtain 
\begin{eqnarray}
&& \psi_x(t) =  \lambda {\rm e}^{-(\lambda+\mu) t -\mu x}\Bigg\{ I_0(2\sqrt{\lambda \mu t (t+x)}) 
 \nonumber 
\\
&&\quad -  \lambda \mu t^2 \sum_{k=0}^{+\infty} \frac{(\lambda \mu t x)^{k}}{k!(k+2)!} \;  {}_{2}F_{1}\left(-k,k+2; k+3; -\frac{t}{x}\right)
 \nonumber
\\
&& \quad -   2 \sum_{k=0}^{+\infty} \sum_{r=0}^{+\infty} \frac{(\lambda \mu)^{k+r+2} t^{2 k+r+4} x^r (2k+1)! }
 {(k+2)!k! r!(2k+r+4)!}  \,{}_{2}F_{1}\left(-r,r+2;2k+r+5; -\frac{t}{x} \right) \Bigg\}.
 \nonumber \\
 \label{psixdue}
\end{eqnarray}
Finally, recalling the integral form of the Gauss Hypergeometric function (see, for instance, 
Eq.\ 15.3.1 of Abramowitz and Stegun \cite{Abr}), and making use of Eq.\ (7.2.1.2) 
of \cite{Prudnikov3} and  Eq.\ (2.15.1.1) of \cite{Prudnikov2}, 
the proof follows from (\ref{psixdue}) after some calculations. 
\end{proof}
\par
We are now able to obtain the PDF of the first renewal cycle when $x\in (0, \infty)$. 
\begin{proposition}
Under the same assumptions of Proposition \ref{propdensTxpos}, the PDF of $C_x$ 
for $y>x$ is given by
\begin{eqnarray}
&& \hspace*{-0.5cm}
f_{C_x}(y)= \frac{1}{2} \lambda e^{-\lambda \frac{y-x}{2}-\mu \frac{y+x}{2}} 
\Bigg\{ I_0\left(\sqrt{\lambda \mu (y^2-x^2)}\right)
\nonumber
\\
&& \hspace*{0.5cm}
+\frac{1}{2} \sum_{r=0}^{+\infty} \frac{(\lambda \mu {x(y-x)}/{2})^r}{r! (r+1)!}
\sum_{j=0}^r {r \choose j} (j+r+1) \left(\frac{y-x}{2 x}\right)^j  
\nonumber
\\
&& \hspace*{0.5cm}
\times \left[-1+ {}_{1}F_{2}\left(-\frac{1}{2};\frac{(j+r+1)}{2},1+\frac{(j+r)}{2};\lambda \mu \left(\frac{y-x}{2}\right)^2\right)\right]
\Bigg\}.
\label{densCxexp}
\end{eqnarray}
\end{proposition}
\begin{proof}
The proof immediately follows from Proposition \ref{propdensTxpos}, and recalling Eq.\ (\ref{relTCx}).
\end{proof}
\par
Some plots of the PDF $f_{C_x}(y)$ are provided in Figure 3. We note that 
$f_{C_x}(x)=\lambda {\rm e}^{-\mu x}/2$, $x\in (0, \infty)$. 
%--------------------------------------------------------------------
\begin{figure}[t]  %%%%%%%%%%% FIGURA  1
%\vspace{20mm}
\begin{center}
\centerline{
\epsfxsize=6.5cm
\epsfbox{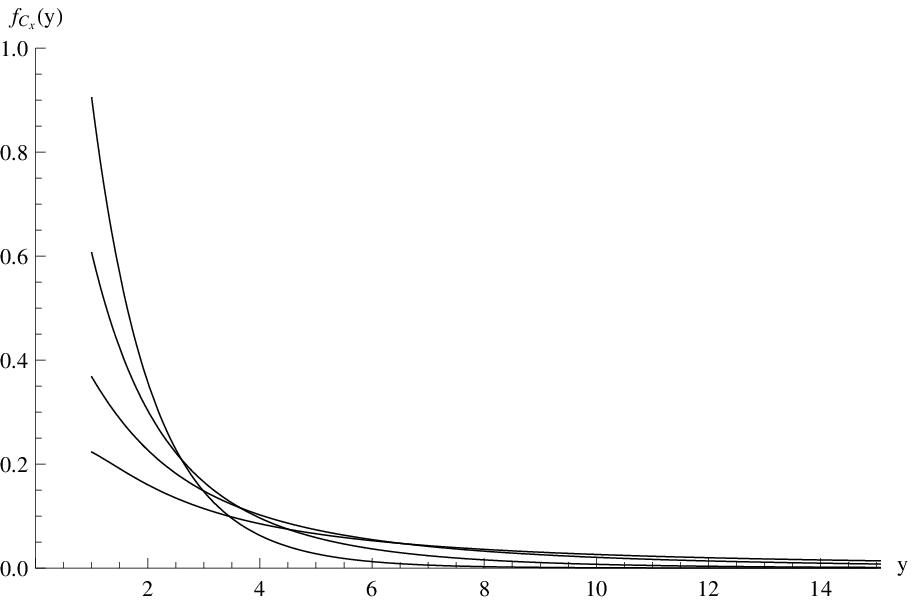}
\hspace*{0.2cm}
\epsfxsize=6.5cm
\epsfbox{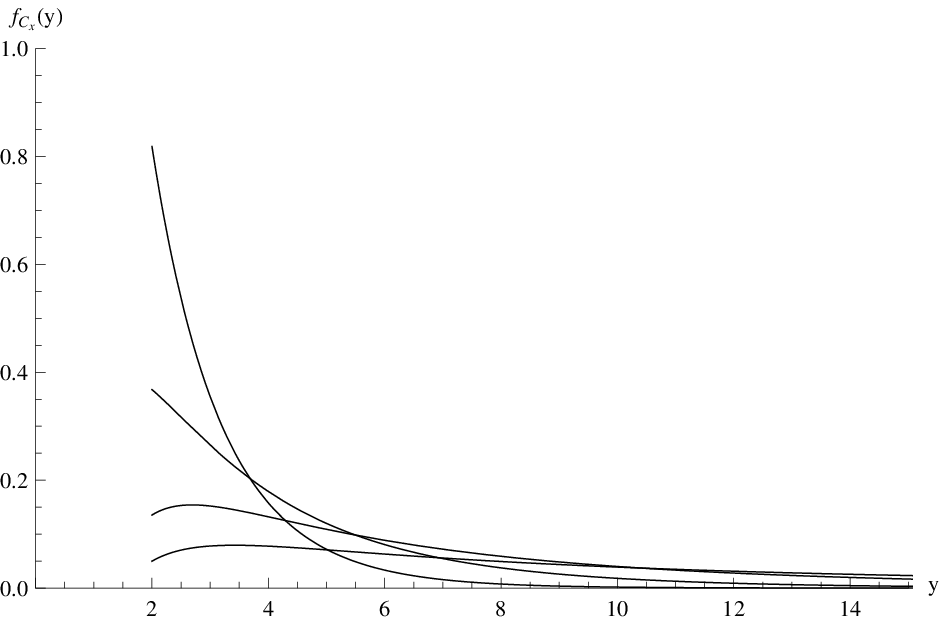}
}
\end{center}
\vspace{-0.5cm}
\caption{Density $f_{C_x}(y)$, given in (\ref{densCxexp}),
for $\lambda=2$, $x=1$ (left-hand side) and $\lambda=2$, $x=2$ (right-hand side) 
with $\mu=0.1$, $0.5$, $1$, $1.5$ from top  to bottom near the origin.}
\end{figure}
%--------------------------------------------------------------------  
%
\begin{remark}
It is not hard to show that if $x\to 0^+$, then the PDF of $T_x$, given in  (\ref{psixexp_pos}), 
tends to  the PDF of $T_0$, shown in (\ref{psi0exp}). Indeed, by virtue of 
Eq.\ (2.15.1.1) of \cite{Prudnikov2}, for any fixed $t\in (0,\infty)$, we have 
\begin{eqnarray}
\lim_{x\to 0^+}\psi_x(t) & =& \lambda e^{-(\lambda + \mu) t}   \Big\{ I_0\left(2 t\sqrt{\lambda \mu}\right)
\nonumber
\\
&-& \int_0^1 \frac{1}{z}\,I_1\left(2 t z \sqrt{\lambda \mu}\right)
I_1\left(2 t (1-z)\sqrt{\lambda \mu}\right){\rm d}z\Big\}.
\nonumber
\end{eqnarray}
The latter expression is identical to (\ref{psi0exp}), due to  Eq.\ (2.15.19.9) of \cite{Prudnikov2} and  the 
following well-known recurrence relation for the Bessel function (see, for instance, (9.6.26) of \cite{Abr}): 
$I_{n-1}(z)-I_{n+1}(z)=(2n/z)I_{n}(z)$. 
\par
Similarly, one can show that if $x\to 0^+$, then the PDF  given in  (\ref{densCxexp}) 
tends to  the PDF   (\ref{fc0}).   
\end{remark}
\par
In the following proposition we obtain the MGF of the first-passage-time defined in 
(\ref{Txdef}).
\begin{proposition}
Under the same assumptions of Proposition \ref{propdensTxpos},  for $s<(\sqrt{\lambda}-\sqrt{\mu})^2$ we have 
\begin{eqnarray}
M_{T_x}(s):=\mathbb E({\rm e}^{s T_x}) 
&= & \frac{\lambda+\mu- s- \sqrt{(\lambda+\mu-s)^2-4 \lambda \mu}}{2 \mu} 
\nonumber \\
&\times & {\rm e}^{\frac{x}{2} \left[\lambda-\mu-s- \sqrt{(\lambda+\mu-s)^2-4 \lambda \mu} \right]}.
\label{fgmTxpos}
\end{eqnarray}
\end{proposition}
\begin{proof}
Recalling (\ref{eq:defLsft}), from Eq.\ (\ref{psixexp_pos}) we have 
\begin{eqnarray*}
&& \hspace*{-0.5cm}
M_{T_x}(s)=\lambda {\rm e}^{-\mu x}  {\cal L}_{\lambda+\mu-s}\left[I_0\left(2 \sqrt{\lambda \mu} \sqrt{t^2+t x}\right)\right]
\nonumber
\\
&& \hspace*{0.7cm} 
+\frac{1}{2} \lambda {\rm e}^{-\mu x}   \sum_{r=0}^{+\infty} \frac{(\lambda \mu)^r}{r! (r+1)!}\sum_{j=0}^r {r \choose j} (j+r+1) x^{r-j} 
\nonumber
\\
&& \hspace*{0.7cm} 
\times {\cal L}_{\lambda+\mu-s}\left[t^{r+j} \left(-1+ {}_{1}F_{2}\left(-\frac{1}{2};\frac{(j+r+1)}{2},1+\frac{(j+r)}{2};\lambda \mu t^2\right)\right)\right].
\end{eqnarray*}
Hence, due to Eqs.\ (3.15.3.1) of \cite{Prudnikov4} and (4.23.17) of \cite{Erdelyi1}, we have for $s<(\sqrt{\lambda}-\sqrt{\mu})^2$
\begin{eqnarray*}
&& \hspace*{-0.5cm}
M_{T_x}(s)=\lambda {\rm e}^{-\mu x} \frac{{\rm e}^{\frac{x}{2} \left[\lambda+\mu-s- \sqrt{(\lambda+\mu-s)^2-4 \lambda \mu} \right]}}
{\sqrt{(\lambda+\mu-s)^2-4 \lambda \mu}}+\frac{1}{2} \lambda {\rm e}^{-\mu x}  \left[-1+\sqrt{1-\frac{4 \lambda \mu}{(\lambda+\mu-s)^2}}    \right]
\nonumber
\\
&& \hspace*{0.7cm} 
\times \frac{1}{\lambda+\mu-s} \sum_{r=0}^{+\infty} \left(\frac{\lambda \mu x}{\lambda+\mu-s} \right)^r \sum_{j=0}^{r} \left(\frac{1}{x(\lambda+\mu-s)}\right)^j {j+r+1 \choose j}
\frac{1}{(r-j)!}.
\end{eqnarray*}
Consequently, due to Eqs.\ (7.2.2.1), (7.2.2.8) and (7.11.1.15) of  \cite{Prudnikov3}, we obtain
\begin{eqnarray*}
&& \hspace*{-0.5cm}
M_{T_x}(s)=\lambda {\rm e}^{-\mu x} \frac{{\rm e}^{\frac{x}{2} \left[\lambda+\mu-s- \sqrt{(\lambda+\mu-s)^2-4 \lambda \mu} \right]}}
{\sqrt{(\lambda+\mu-s)^2-4 \lambda \mu}}+\frac{1}{2} \lambda {\rm e}^{-\mu x}  \left[-1+\sqrt{1-\frac{4 \lambda \mu}{(\lambda+\mu-s)^2}}    \right]
\nonumber
\\
&& \hspace*{0.7cm} 
\times \frac{1}{\lambda+\mu-s} {\rm e}^{\lambda \mu x/(\lambda+\mu-s)}
\sum_{j=0}^{+\infty} \left[\frac{\lambda \mu}{(\lambda+\mu-s)^2}  \right]^j {\mathbf L}_{j}^{j+1} \left(-\frac{\lambda \mu x}{\lambda+\mu-s}  \right),
\end{eqnarray*}
where ${\mathbf L}_{n}^{\beta}$, $n\in \mathbb{N}$, denotes the generalized Laguerre polynomials. 
Finally, recalling Eq.\ (5.11.4.7) of  \cite{Prudnikov2}, it is
$$
M_{T_x}(s)=\frac{4 \lambda \mu-\left[\lambda+\mu-s-\sqrt{(\lambda+\mu-s)^2-4 \lambda \mu} \right]^2}{4 \mu \left[\sqrt{(\lambda+\mu-s)^2}-4 \lambda \mu\right]}\,
{\rm e}^{\frac{x}{2} \left[\lambda-\mu-s- \sqrt{(\lambda+\mu-s)^2-4 \lambda \mu} \right]},
$$
so that Eq.\ (\ref{fgmTxpos}) immediately follows.
\end{proof}
\par
Hereafter we obtain the MGF of the absorption time at the origin.
\begin{proposition}
Under the same assumptions of Proposition \ref{propdensTxpos},
for $s<(\sqrt{\lambda}-\sqrt{\mu})^2/{2}$, the MGF of $A_x$ is 
\begin{equation}
M_{A_x}(s):=\mathbb E({\rm e}^{s A_x})= \frac{2 \alpha \lambda {\rm e}^{\frac{x}{2} \left[\lambda-\mu- \sqrt{(\lambda+\mu-2s)^2-4 \lambda \mu} \right]}}{2 \lambda (\alpha-1)+(\lambda+\mu-2 s)+\sqrt{(\lambda+\mu-2 s)^2 -4 \lambda \mu}}.
\label{fgmAx}
\end{equation}
\label{proposAxpos}
\end{proposition}
\begin{proof}
From Eqs.\ (\ref{distrM}) and (\ref{RelAx}), we have the following relation: 
$$
M_{A_x}(s)=\frac{\alpha M_{C_x}(s)}{1+(\alpha-1) M_{C_0}(s)}.
$$
Hence, recalling (\ref{fgmC0}) and noting that  
\begin{eqnarray}
 M_{C_x}(s) & :=& \mathbb E({\rm e}^{s C_x})=\frac{\lambda+\mu- 2s- \sqrt{(\lambda+\mu-2s)^2-4 \lambda \mu}}{2 \mu}
 \nonumber \\
 & \times& {\rm e}^{\frac{x}{2} \left[\lambda-\mu- \sqrt{(\lambda+\mu-2s)^2-4 \lambda \mu} \right]}
\label{fgmMCxs}
\end{eqnarray}
for $s<(\sqrt{\lambda}-\sqrt{\mu})^2/2$, Eq.\ (\ref{fgmAx}) follows after some calculations. 
\end{proof}
\par
Let us now determine the moments of the renewal cycle when the initial state is non-zero.
\begin{proposition}\label{prop:momCxn}
Under the same assumptions of Proposition \ref{propdensTxpos}, for $n\in \mathbb N$ 
the $n$th moment of $C_x$ is given by
\begin{eqnarray}
&& \hspace*{-0.2cm}
\mathbb E(C_x^n)=\frac{\lambda}{\lambda+\mu} {\rm e}^{\frac{x}{2} (\lambda-\mu) } \frac{2^n}{(\lambda+\mu)^n}
 \sum_{h=0}^{n} \left(-\frac{\lambda+\mu}{(\sqrt{\lambda}-\sqrt{\mu})^2} \right)^h 
 \nonumber \\
 && \hspace*{0.4cm}
\times \; {}_{2}F_{1}\left(\frac{1+n-h}{2},\frac{2+n-h}{2};2;\frac{4\lambda \mu}{(\lambda+\mu)^2} \right)
\nonumber
\\
&& \hspace*{0.4cm}
\times \sum_{j=0}^{+\infty} \left[-\frac{(\lambda-\mu) x}{2} \right]^j \frac{1}{j!}
{j/2 \choose h}  \; {}_{2}F_{1}\left(-h,-\frac{j}{2};\frac{j}{2}+1-h; \left(\frac{\sqrt{\lambda}-\sqrt{\mu}}{\sqrt{\lambda}+\sqrt{\mu}}\right)^2 \right),
\nonumber 
\\
\label{momCx}
\end{eqnarray}
with ${}_{2}F_{1}$ given in (\ref{Hyper2F1}), and 
${x \choose h}:=x(x-1)(x-2)\ldots(x-h+1)/h!$ for $x\in \mathbb{R}$ and $h\in \mathbb{N}$. 
\end{proposition}
\begin{proof}
Comparing the MGFs (\ref{fgmC0}) and (\ref{fgmMCxs}), for $s<(\sqrt{\lambda}-\sqrt{\mu})^2/{2}$ 
we have
\begin{equation}
 M_{C_x}(s)=M_{C_0}(s)\cdot {\rm e}^{\frac{x}{2} \left[\lambda-\mu- \sqrt{(\lambda+\mu-2s)^2-4 \lambda \mu} \right]}.
 \label{eq:relFG}
\end{equation}
We note that
\begin{eqnarray}
&& \hspace*{-0.5cm}
{\rm e}^{\frac{x}{2} \left[\lambda-\mu- \sqrt{(\lambda+\mu-2s)^2-4 \lambda \mu} \right]}=\sum_{n=0}^{+\infty} \frac{x^n}{2^n n!}\left[\lambda-\mu- \sqrt{(\lambda+\mu-2s)^2-4 \lambda \mu} \right]^n
\nonumber
\\
&& \hspace*{2.5cm}
=\sum_{n=0}^{+\infty} \frac{x^n}{2^n n!} \sum_{j=0}^n {n \choose j} (\lambda-\mu)^{n-j}(-1)^j [(\lambda+\mu-2s)^2-4 \lambda \mu]^{j/2},
%\label{serie1}
\nonumber
\end{eqnarray}
where
\begin{eqnarray*}
&&  
 [(\lambda+\mu-2s)^2-4 \lambda \mu]^{j/2}
 \\
&& =(\lambda-\mu)^j \sum_{k=0}^{+\infty} {j/2 \choose k} \left(-\frac{2 s}{(\sqrt{\lambda}-\sqrt{\mu})^2} \right)^k
 \sum_{l=0}^{+\infty} {j/2 \choose l} \left(-\frac{2 s}{(\sqrt{\lambda}+\sqrt{\mu})^2} \right)^l
\\
&& % \hspace*{1.5cm}
=(\lambda-\mu)^j \sum_{r=0}^{+\infty} s^r  \sum_{h=0}^{r} {j/2 \choose h}  {j/2 \choose r-h}  \left(-\frac{2}{(\sqrt{\lambda}-\sqrt{\mu})^2} \right)^{r-h}
\left(-\frac{2}{(\sqrt{\lambda}+\sqrt{\mu})^2} \right)^h,
\end{eqnarray*}
so that 
\begin{eqnarray}
&& {\rm e}^{\frac{x}{2} \left[\lambda-\mu- \sqrt{(\lambda+\mu-2s)^2-4 \lambda \mu} \right]}
\nonumber \\
&& ={\rm e}^{\frac{x}{2} (\lambda-\mu) }
\sum_{r=0}^{+\infty} s^r \left(-\frac{2}{(\sqrt{\lambda}-\sqrt{\mu})^2} \right)^{r}
\sum_{j=0}^{+\infty} \left[-\frac{(\lambda-\mu) x}{2} \right]^j \frac{1}{j!}
\nonumber
\\
&& %\hspace*{1.5cm}
\times \sum_{h=0}^r {j/2 \choose h}  {j/2 \choose r-h} \left(\frac{\sqrt{\lambda}-\sqrt{\mu}}{\sqrt{\lambda}+\sqrt{\mu}}   \right)^{2 h}
={\rm e}^{\frac{x}{2} (\lambda-\mu) }
\sum_{r=0}^{+\infty} s^r \left(-\frac{2}{(\sqrt{\lambda}-\sqrt{\mu})^2} \right)^{r}
\nonumber
\\
&& %\hspace*{1.5cm}
\times
\sum_{j=0}^{+\infty} \left[-\frac{(\lambda-\mu) x}{2} \right]^j \frac{1}{j!}
{j/2 \choose r} \;{}_{2}F_{1}\left(-r,-j/2;j/2+1-r; \left(\frac{\sqrt{\lambda}-\sqrt{\mu}}{\sqrt{\lambda}+\sqrt{\mu}}\right)^2 \right).
\nonumber \\
\label{serie2}
\end{eqnarray}
Hence, the moments of $C_x$ can be obtained from (\ref{eq:relFG}) and 
taking into account Eqs.\ (\ref{momC0}) and (\ref{serie2}), after some calculations.
\end{proof}
\par
We can now provide the moments of $A_x$. 
\begin{proposition}\label{prop:momAxn}
Under the same assumptions of Proposition \ref{propdensTxpos}, the $n$th moment of $A_x$, 
for $n\in \mathbb N$, is given by
\begin{eqnarray}
&& \hspace*{-0.2cm}
\mathbb E(A_x^n)=2 \alpha \lambda {\rm e}^{\frac{x}{2} (\lambda-\mu) }
 \sum_{h=0}^{n} \left(-\frac{2}{(\sqrt{\lambda}-\sqrt{\mu})^2} \right)^h \frac{(8 \lambda (\alpha-1))^{n-h}}{(4 \lambda \alpha (\mu+ \lambda (\alpha-1)))^{n-h+1}}
\nonumber
\\
&& \hspace*{0.4cm}
\times \left[ 2 \mu+2 \lambda (\alpha-1)+\frac{2 \lambda \mu}{\lambda+\mu} \sum_{m=1}^{n-h} \left(\frac{\alpha \mu+\alpha \lambda (\alpha-1) }{(\alpha-1)(\lambda+\mu)} \right)^m\; \right.
\nonumber
\\
&& \hspace*{0.4cm}
\times \; \left.
 {}_{2}F_{1}\left(\frac{m+1}{2},\frac{m+2}{2};2; \frac{4\lambda \mu}{(\lambda+\mu)^2} \right)
\right]
\nonumber
\\
&& \hspace*{0.4cm}
\times \sum_{j=0}^{+\infty} \left[-\frac{(\lambda-\mu) x}{2} \right]^j \frac{1}{j!}
{j/2 \choose h}\; {}_{2}F_{1}\left(-h,-\frac{j}{2};\frac{j}{2}+1-h; \left(\frac{\sqrt{\lambda}-\sqrt{\mu}}{\sqrt{\lambda}+\sqrt{\mu}}\right)^2 \right).
\nonumber \\
&& 
\label{momAx}
\end{eqnarray}
\end{proposition}
\begin{proof}
Due to Eqs.\  (\ref{fgmA0}) and (\ref{fgmAx}), the following relation holds:
$$
 M_{A_x}(s)=M_{A_0}(s)\cdot {\rm e}^{\frac{x}{2} \left[\lambda-\mu- \sqrt{(\lambda+\mu-2s)^2-4 \lambda \mu} \right]}.
$$
The moments (\ref{momAx}) then follow from Eqs.\ (\ref{momA0}) and (\ref{serie2}), similarly as in the proof of 
Proposition \ref{prop:momCxn}.
\end{proof}
\par
From Propositions \ref{prop:momCxn} and \ref{prop:momAxn} the following results immediately follow. 
\begin{proposition}
For $0<\mu<\lambda$ and $\alpha\in (0,1)$, the means of $C_x$ and $A_x$ are 
$$
\mathbb E(C_x)=\frac{2+(\lambda+\mu)x}{\lambda-\mu},
\qquad 
\mathbb E(A_x)=\frac{2+\alpha(\lambda+\mu)x}{\alpha(\lambda-\mu)},
$$
whereas their variances are given by
$$
Var(C_x)=\frac{4 \mu}{(\lambda-\mu)^3}-\frac{2(\lambda^2-\mu^2-2 \lambda \mu)x}{(\lambda-\mu)^3}-\frac{(\lambda+\mu)^2 x^2}{2 (\lambda-\mu)^2},
$$
$$
Var(A_x)=\frac{4 \mu}{\alpha (\lambda-\mu)^3}-\frac{2(\lambda^2-\mu^2-2 \alpha \lambda \mu)x}{\alpha (\lambda-\mu)^3}-\frac{(\lambda+\mu)^2 x^2}{2 (\lambda-\mu)^2}.
$$
\end{proposition}
\par
It is easy to see that $\mathbb E(A_x)$ is decreasing in $\alpha$, and clearly tends to $\mathbb E(C_x)$ as 
$\alpha\to 1^-$. Indeed, $A_x$ identifies with $C_x$ when $\alpha=1$.

%--------------------------------------------------------------------
\section{\bf  Conditional distribution of the process within a renewal cycle $C_0$}\label{section:5}
%--------------------------------------------------------------------
In this section we derive the conditional distribution of $X(t)$ within a renewal cycle in the case 
of zero initial state. Specifically, let us consider renewal cycles that start with $X(0)=0$ and ends at $C_0$. 
We recall that $Y(t)$ is the compound Poisson process defined in (\ref{Ydef}) and $T_0$ is the 
stopping time introduced in (\ref{Txdef}) for $x=0$. Given $T_0$, we consider 
any sample path of $Y(t)$ which crosses the boundary $\{ B(t)=t,\; t>0\}$ at $T_0$. 
For any given $t\in (0,C_0)$, let $W(t)$ be the time coordinate at which the sample path of $Y(t)$ 
crosses the line $\{L_t(w)=t-w,\;  w\in (0,t)\}$. The value of $X(t)$ within a renewal cycle is then $X(t)=2 W(t)-t$, 
in the case of zero initial state. 
Notice that $W(t)$ is the total time in $(0,t]$ at which the telegraph process is moving upwards.  
As example, Figures 4 and 5 provide sample paths of such processes. 
We observe that, for every $t\in (0,C_0)$ and given $T_0=C_0/2$, it results $t/2<W(t)\leq T_0$. 
% ===================================
\begin{figure}[t]
\begin{center}
\begin{picture}(341,236) 
\put(20,30){\vector(1,0){300}} 
\put(40,5){\vector(0,1){200}} 
\put(40,30.5){\line(1,0){60}} 
\put(100,30.5){\line(0,1){24.5}} 
\put(100,55){\line(1,0){50}} 
\put(150,55){\line(0,1){85}} 
\put(40,30){\line(1,1){10}} 
\put(55,45){\line(1,1){10}} 
\put(70,60){\line(1,1){10}} 
\put(85,75){\line(1,1){10}} 
\put(100,90){\line(1,-1){10}} 
\put(115,75){\line(1,-1){10}} 
\put(130,70){\line(1,1){10}} 
\put(145,85){\line(1,1){10}} 
\put(160,100){\line(1,1){10}} 
\put(175,115){\line(1,-1){10}} 
\put(190,100){\line(1,-1){10}} 
\put(205,85){\line(1,-1){10}} 
\put(220,70){\line(1,-1){10}} 
\put(235,55){\line(1,-1){10}} 
\put(250,40){\line(1,-1){10}} 
\put(105,95){\circle*{1}} 
\put(110,100){\circle*{1}} 
\put(115,105){\circle*{1}} 
\put(120,110){\circle*{1}} 
\put(125,115){\circle*{1}} 
\put(130,120){\circle*{1}} 
\put(135,125){\circle*{1}} 
\put(140,130){\circle*{1}} 
\put(145,135){\circle*{1}} 
\put(150,140){\circle*{1}} 
\put(155,145){\circle*{1}} 
\put(160,150){\circle*{1}} 
\put(165,155){\circle*{1}} 
\put(170,160){\circle*{1}} 
\put(175,165){\circle*{1}} 
\put(100,27){\line(0,1){3}} 
\put(125,27){\line(0,1){3}} 
\put(150,27){\line(0,1){3}} 
\put(175,27){\line(0,1){3}} 
\put(260,27){\line(0,1){3}} 
\put(125,30){\circle*{1}} 
\put(120,35){\circle*{1}} 
\put(115,40){\circle*{1}} 
\put(110,45){\circle*{1}} 
\put(105,50){\circle*{1}} 
\put(100,55){\circle*{1}} 
\put(95,60){\circle*{1}} 
\put(90,65){\circle*{1}} 
\put(85,70){\circle*{1}} 
\put(80,75){\circle*{1}} 
\put(75,80){\circle*{1}} 
\put(70,85){\circle*{1}} 
\put(65,90){\circle*{1}} 
\put(60,95){\circle*{1}} 
\put(55,100){\circle*{1}} 
\put(50,105){\circle*{1}} 
\put(45,110){\circle*{1}} 
\put(40,115){\circle*{1}} 
\put(175,30){\circle*{1}} 
\put(170,35){\circle*{1}} 
\put(165,40){\circle*{1}} 
\put(160,45){\circle*{1}} 
\put(155,50){\circle*{1}} 
\put(150,55){\circle*{1}} 
\put(145,60){\circle*{1}} 
\put(140,65){\circle*{1}} 
\put(135,70){\circle*{1}} 
\put(130,75){\circle*{1}} 
\put(125,80){\circle*{1}} 
\put(120,85){\circle*{1}} 
\put(115,90){\circle*{1}} 
\put(110,95){\circle*{1}} 
\put(105,100){\circle*{1}} 
\put(100,105){\circle*{1}} 
\put(95,110){\circle*{1}} 
\put(90,115){\circle*{1}} 
\put(85,120){\circle*{1}} 
\put(80,125){\circle*{1}} 
\put(75,130){\circle*{1}} 
\put(70,135){\circle*{1}} 
\put(65,140){\circle*{1}} 
\put(60,145){\circle*{1}} 
\put(55,150){\circle*{1}} 
\put(50,155){\circle*{1}} 
\put(45,160){\circle*{1}} 
\put(40,165){\circle*{1}} 
\put(80,6){\makebox(40,15)[t]{{$t_1$}}}
\put(105,6){\makebox(40,15)[t]{{$t_2$}}}
\put(155,6){\makebox(40,15)[t]{{$t_3$}}}
\put(130,6){\makebox(40,15)[t]{{$T_0$}}}
\put(235,6){\makebox(50,15)[t]{{$C_0$}}} 
\put(12,6){\makebox(40,15)[t]{0}} 
\put(235,200){\line(1,0){10}} 
\put(250,200){\line(1,0){10}} 
\put(265,200){\line(1,0){10}} 
\put(280,200){\line(1,0){10}} 
\put(235,180){\line(1,0){55}} 
\put(290,190){\makebox(40,15)[t]{{$X(t)$}}}
\put(290,170){\makebox(40,15)[t]{{$Y(t)$}}}
\put(300,5){\makebox(40,15)[t]{{$t$}}}
\put(170,165){\makebox(40,15)[t]{{$B(t)=t$}}}
\put(-20,155){\makebox(40,15)[t]{{$L_{t_3}(w)=t_3-w$}}}
\put(-20,105){\makebox(40,15)[t]{{$L_{t_2}(w)=t_2-w$}}}
\end{picture} 
\end{center}
\vspace{-0.5cm}
\caption{Sample paths of the processes $X(t)$ and $Y(t)$, with $x=0$.}
\end{figure}
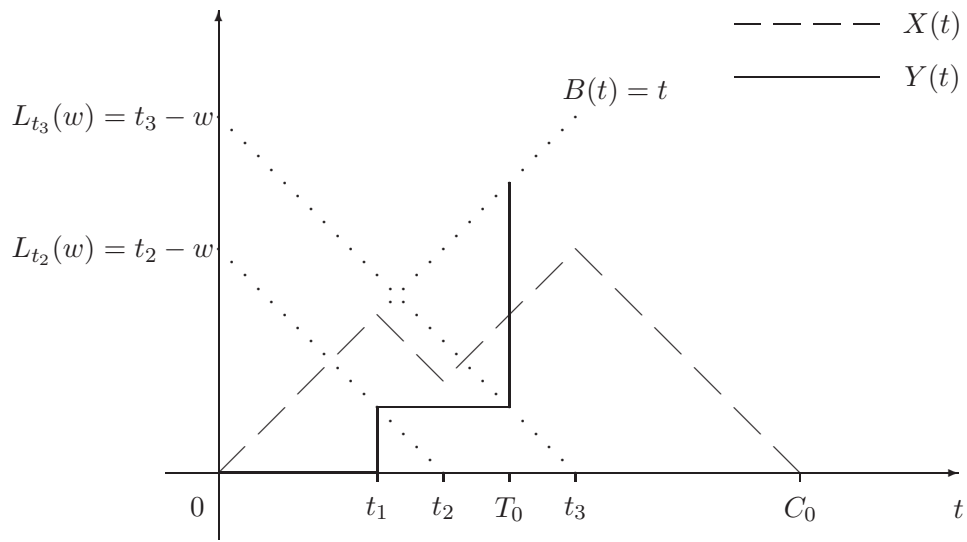
%-------------------------------------------------------------------------------
% ===================================
\begin{figure}[t]
\begin{center}
\begin{picture}(341,236) 
\put(20,30){\vector(1,0){300}} 
\put(40,5){\vector(0,1){200}} 
\put(40,30){\line(1,1){60}} 
\put(100,90){\line(1,0){25}} 
\put(125,90){\line(1,1){50}} 
\put(175,140){\line(1,0){85}} 
%
%\put(105,95){\circle*{1}} 
%\put(110,100){\circle*{1}} 
%\put(115,105){\circle*{1}} 
%\put(120,110){\circle*{1}} 
%\put(125,115){\circle*{1}} 
%\put(130,120){\circle*{1}} 
%\put(135,125){\circle*{1}} 
%\put(140,130){\circle*{1}} 
%\put(145,135){\circle*{1}} 
\put(150,140){\circle*{1}} 
\put(155,140){\circle*{1}} 
\put(160,140){\circle*{1}} 
\put(165,140){\circle*{1}} 
\put(170,140){\circle*{1}} 
\put(145,140){\circle*{1}} 
\put(140,140){\circle*{1}} 
\put(135,140){\circle*{1}} 
\put(130,140){\circle*{1}} 
\put(125,140){\circle*{1}} 
\put(120,140){\circle*{1}} 
\put(115,140){\circle*{1}} 
\put(110,140){\circle*{1}} 
\put(105,140){\circle*{1}} 
\put(100,140){\circle*{1}} 
\put(95,140){\circle*{1}} 
\put(90,140){\circle*{1}} 
\put(85,140){\circle*{1}} 
\put(80,140){\circle*{1}} 
\put(75,140){\circle*{1}} 
\put(70,140){\circle*{1}} 
\put(65,140){\circle*{1}} 
\put(60,140){\circle*{1}} 
\put(55,140){\circle*{1}} 
\put(50,140){\circle*{1}} 
\put(45,140){\circle*{1}} 
\put(40,140){\line(-1,0){3}} 
\put(100,27){\line(0,1){3}} 
\put(125,27){\line(0,1){3}} 
\put(175,27){\line(0,1){3}} 
\put(260,27){\line(0,1){3}} 
\put(5,130){\makebox(40,15)[t]{{$T_0$}}}
\put(80,6){\makebox(40,15)[t]{{$t_1$}}}
\put(105,6){\makebox(40,15)[t]{{$t_2$}}}
\put(155,6){\makebox(40,15)[t]{{$t_3$}}}
\put(235,6){\makebox(50,15)[t]{{$C_0$}}} 
\put(12,6){\makebox(40,15)[t]{0}} 
\put(300,5){\makebox(40,15)[t]{{$t$}}}
\put(100,35){\circle*{1}} 
\put(100,40){\circle*{1}} 
\put(100,45){\circle*{1}} 
\put(100,50){\circle*{1}} 
\put(100,55){\circle*{1}} 
\put(100,60){\circle*{1}} 
\put(100,65){\circle*{1}} 
\put(100,70){\circle*{1}} 
\put(100,75){\circle*{1}} 
\put(100,80){\circle*{1}} 
\put(100,85){\circle*{1}} 
\put(100,90){\circle*{1}} 
\put(125,35){\circle*{1}} 
\put(125,40){\circle*{1}} 
\put(125,45){\circle*{1}} 
\put(125,50){\circle*{1}} 
\put(125,55){\circle*{1}} 
\put(125,60){\circle*{1}} 
\put(125,65){\circle*{1}} 
\put(125,70){\circle*{1}} 
\put(125,75){\circle*{1}} 
\put(125,80){\circle*{1}} 
\put(125,85){\circle*{1}} 
\put(125,90){\circle*{1}} 
\put(175,35){\circle*{1}} 
\put(175,40){\circle*{1}} 
\put(175,45){\circle*{1}} 
\put(175,50){\circle*{1}} 
\put(175,55){\circle*{1}} 
\put(175,60){\circle*{1}} 
\put(175,65){\circle*{1}} 
\put(175,70){\circle*{1}} 
\put(175,75){\circle*{1}} 
\put(175,80){\circle*{1}} 
\put(175,85){\circle*{1}} 
\put(175,95){\circle*{1}} 
\put(175,90){\circle*{1}} 
\put(175,100){\circle*{1}} 
\put(175,105){\circle*{1}} 
\put(175,110){\circle*{1}} 
\put(175,115){\circle*{1}} 
\put(175,120){\circle*{1}} 
\put(175,125){\circle*{1}} 
\put(175,130){\circle*{1}} 
\put(175,135){\circle*{1}} 
\put(175,140){\circle*{1}} 
\put(260,35){\circle*{1}} 
\put(260,40){\circle*{1}} 
\put(260,45){\circle*{1}} 
\put(260,50){\circle*{1}} 
\put(260,55){\circle*{1}} 
\put(260,60){\circle*{1}} 
\put(260,65){\circle*{1}} 
\put(260,70){\circle*{1}} 
\put(260,75){\circle*{1}} 
\put(260,80){\circle*{1}} 
\put(260,85){\circle*{1}} 
\put(260,95){\circle*{1}} 
\put(260,90){\circle*{1}} 
\put(260,100){\circle*{1}} 
\put(260,105){\circle*{1}} 
\put(260,110){\circle*{1}} 
\put(260,115){\circle*{1}} 
\put(260,120){\circle*{1}} 
\put(260,125){\circle*{1}} 
\put(260,130){\circle*{1}} 
\put(260,135){\circle*{1}} 
\put(260,140){\circle*{1}} 
\end{picture} 
\end{center}
\vspace{-0.5cm}
\caption{Sample path of $W(t)$ corresponding to the case of Figure 4.}
\end{figure}
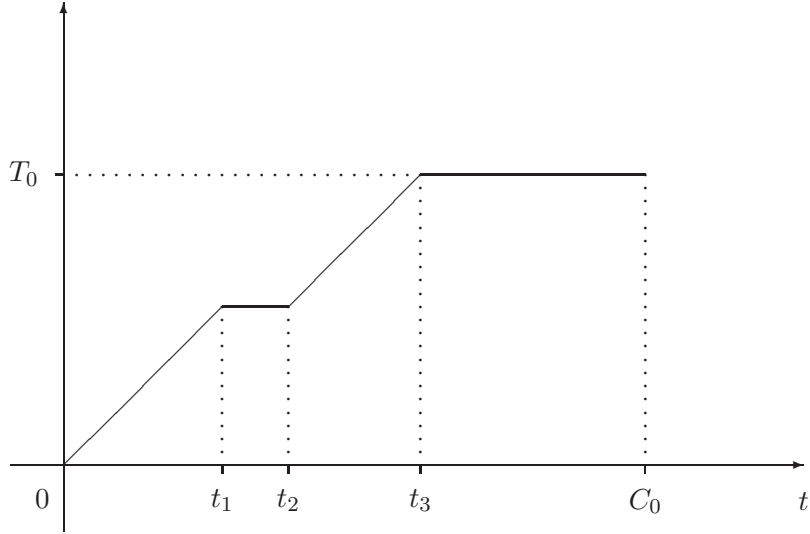
%-------------------------------------------------------------------------------
\par
Let us now determine the subdistribution function of $Y(t)$ and $T_0$, in the case of zero initial state. 
For  any $t\in (0, \infty)$ it is defined as  
\begin{equation}
 F_{Y(t),T_0}(y,\tau){\rm d}\tau:=\mathbb P[Y(t)\leq y, T_0\in {\rm d}\tau], 
 \qquad  y\in (0,t), \;\;  \tau\in (t, \infty).
 \label{eq:defcdfYT0}
\end{equation}
\begin{proposition}
Under assumptions (\ref{defF}) and (\ref{defG}), for $0<\mu <\lambda$, the subdistribution 
function defined in (\ref{eq:defcdfYT0}) is given by 
\begin{eqnarray} 
&& \hspace*{-1.2cm}
F_{Y(t),T_0}(y,\tau)=\lambda {\rm e}^{-(\lambda+\mu) \tau}\left\{ I_0\left(2\sqrt{\lambda \mu \tau (t-\tau)}\right)\right.
\nonumber \\
&&  \hspace*{-0.9cm}
+\frac{1}{2} \sum_{r=0}^{+\infty} \frac{[\lambda \mu t (\tau-t)]^r}{r! (r+1)!}  
 \sum_{j=0}^r {r\choose j} (j+r+1) \left(\frac{\tau-t}{t}\right)^j
\nonumber
\\
&&\hspace*{-0.9cm}
\times 
\left[-1+{}_{1}F_{2}\left(-\frac{1}{2};\frac{j+r+1}{2},1+\frac{j+r}{2};\lambda \mu (\tau-t)^2\right)\right] 
\nonumber
\\
&& \hspace*{-0.9cm}
+\lambda\,\mu\, y\, \sum_{j=0}^{+\infty} \frac{[\lambda \mu y (\tau-t)]^j}{j!} \;
{}_{0}F_{1}\left(; j+1; \lambda \mu (t-\tau) (y-\tau)\right) 
\nonumber\\
&& \hspace*{-0.9cm}
\times \left[\frac{t}{(j+1)!}\; {}_{1}F_{2}\left(1;j+2,2;\lambda \mu t y\right) -\frac{y}{(j+2)!}\; {}_{1}F_{2}\left(2;j+3,2;\lambda \mu t y\right)\right]
\nonumber
\\
&& \hspace*{-0.9cm}
+\frac{\lambda\,\mu}{2}\sum_{r=0}^{+\infty} \frac{[\lambda \mu (\tau-t)]^r}{r!(r+1)!} 
 \sum_{s=0}^r {r\choose r-s} (2r+1-s) (\tau-t)^{r-s}
\nonumber
\\
&& \hspace*{-0.9cm}
\times \left[-1+{}_{1}F_{2}\left(-\frac{1}{2};r-\frac{s-1}{2},r+1-\frac{s}{2};\lambda \mu (\tau-t)^2\right)\right]
\nonumber
\\
&& \hspace*{-0.9cm}
\left.
\times \sum_{k=0}^{s+1} {s+1 \choose k} (t-y)^{s+1-k} \frac{y^{k+1}}{k+1} \; 
{}_{1}F_{2}\left(1;k+2,2;\lambda \mu t y \right)\right\},
\label{jointdistr4}
\end{eqnarray}
where 
\begin{equation}
{}_{0}F_{1}\left(; b; z\right)=\sum_{n=0}^{+\infty} \frac{z^n}{(b)_n n!} 
 \label{Hyper0F1}
\end{equation}
and  ${}_{1}F_{2}(a;b,c;z)$ is defined in Eq.\ (\ref{Hyper1F2}).
\label{prop_jointdistr}
\end{proposition}
\begin{proof}
Due to (\ref{eq:defcdfYT0}), we note that, for $t>0$, 
\begin{equation}
F_{Y(t),T_0}(y,\tau) 
=\mathbb P[Y(t)=0,T_0\in {\rm d}\tau]/{\rm d}\tau +\int_{0}^{y} p_{Y(t),T_0}(x,\tau) {\rm d}x,
\label{jointdistr1}
\end{equation}
where $p$ is the subdensity $p_{Y(t),T_0}(x,\tau):=\frac{\partial}{\partial x} F_{Y(t),T_0}(x,\tau)$. 
We point out that for $\tau\in (t, \infty)$ and $ y\in (0,t)$, it is
\begin{equation}
\mathbb P[Y(t)=0,T_0\in {\rm d}\tau]/{\rm d}\tau={\rm e}^{-\lambda t} \psi _{t}(\tau-t),
 \label{jointdistr2}
\end{equation}
since  $\mathbb P[Y(t)=0]=\mathbb P[U_1>t]={\rm e}^{-\lambda t}$, $t>0$, and 
$\mathbb P[T_0\in {\rm d}\tau\,|\,Y(t)=0]  = \mathbb P(t+T_t\in {\rm d}\tau)$, $\tau >t$. 
By a similar reasoning, one also has 
\begin{equation}
p_{Y(t),T_0}(x,\tau)=g_0(x,t) \psi_{t-x}(\tau-t),
 \label{jointdistr3}
\end{equation}
where  $g_0$ and $\psi _{t-x}$ are defined in Eqs.\ (\ref{gxdef}) and (\ref{psix}), respectively. 
Making use of (\ref{jointdistr2}) and (\ref{jointdistr3}) in (\ref{jointdistr1}), the function 
$F_{Y(t),T_0}(y,\tau)$ can be obtained by recalling the expressions of $g_0(y,t)$ and $\psi _{x}(t)$ 
provided by (\ref{g0def}), (\ref{hexp})  and (\ref{psixexp_pos}). The resulting expression of 
$F_{Y(t),T_0}(y,\tau)$ involves the following identities:\\ 
\begin{eqnarray*}
&& \hspace*{-0.8cm}
\int_{0}^{y} \frac{(t-x)^{r-j+1}}{\sqrt{x}} I_1 \left(2 \sqrt{\lambda \mu t x}\right) {\rm d}x
\\
&& \hspace*{-0.2cm}
=\sqrt{\lambda \mu t} \sum_{k=0}^{r-j+1} {r-j+1 \choose k} (t-y)^{r-j+1-k} \frac{y^{k+1}}{k+1} \; 
{}_{1}F_{2}\left(1;k+2,2;\lambda \mu t y \right) 
\end{eqnarray*}
and 
\begin{eqnarray*}
&& \hspace*{-1.0cm}
 \frac{\lambda \sqrt{\lambda \mu}}{\sqrt{t}} {\rm e}^{-(\lambda+\mu) \tau}
\int_{0}^{y} \frac{t-x}{\sqrt{x}} 
I_1 \left(2 \sqrt{\lambda \mu t x}\right)I_0 \left(2 \sqrt{\lambda \mu (\tau-t) (\tau-x)}\right) {\rm d}x
\\
&& \hspace*{-0.8cm}
=\lambda^2 \mu {\rm e}^{-(\lambda+\mu) \tau}  
\times 
\sum_{r=0}^{+\infty} \frac{[\lambda \mu (\tau-t)]^r}{r!^2} \sum_{j=0}^r {r \choose j} (\tau-y)^{r-j} \frac{y^{j+1}}{j+1}
\\
&& \hspace*{-0.8cm}
\times 
 \left\{
t  \;{}_{1}F_{2}\left(1;j+2,2;\lambda \mu t y\right)-\frac{y}{j+2} \;{}_{1}F_{2}\left(2;j+3,2;\lambda \mu t y\right) \right\}
\\
&& \hspace*{-1cm}
= \lambda^2\,\mu\, y\, {\rm e}^{-(\lambda+\mu) \tau}\,  \sum_{j=0}^{+\infty} \frac{[\lambda \mu y (\tau-t)]^j}{j!} \; 
{}_{0}F_{1}\left(; j+1; \lambda \mu (t-\tau) (y-\tau)\right)\nonumber\\
&& \hspace*{-0.9cm}
\times \left[\frac{t}{(j+1)!}\; {}_{1}F_{2}\left(1;j+2,2;\lambda \mu t y\right) -\frac{y}{(j+2)!} \; 
{}_{1}F_{2}\left(2;j+3,2;\lambda \mu t y\right)\right],
\end{eqnarray*}
the latter being due to Eq.\ (2.15.2.5) of \cite{Prudnikov2} and identity
$$
I_0 \left(2 \sqrt{\lambda \mu (\tau-t) (\tau-x)}\right)
=\sum_{r=0}^{+\infty} \frac{[\lambda \mu (\tau-t)]^r}{r!^2}
\sum_{j=0}^{r} {r \choose j} (\tau-y)^{r-j} (y-x)^j.
$$
The proof thus follows after some calculations.
\end{proof}
\par
We conclude this paper by giving the expression of the conditional distribution of $X(t)$ given $T_0$, 
within $C_0$. The proof is omitted since it immediately follows from the definition of $W(t)$. 
\begin{proposition}
The conditional distribution of $X(t)$ given $T_0$, during a renewal cycle $C_0$, is expressed as   
\begin{equation}
\mathbb P[X(t)\leq x\,|\,T_0=\tau]
 =\mathbb P[W(t)>t/2\,|\,T_0=\tau]-\mathbb P[W(t)>(t+x)/2\,|\,T_0=\tau],
 \label{eq:PXcondT0}
\end{equation}
for $t\in (0,\tau)$ and $ x\in [0,t]$, where 
$$
\mathbb P[W(t)>w\,|\,T_0=\tau]=\frac{F_{Y(w),T_0}(t-w,\tau)}{\psi_0(\tau)},
\qquad w\in \Big(\frac{t}{2},t\Big),
$$
with $\psi_0(x)$ and $F_{Y(w),T_0}(y,\tau)$ given in (\ref{psi0exp}) and (\ref{jointdistr4}), 
respectively.
\end{proposition}
\par
We omit the explicit expression of the distribution  (\ref{eq:PXcondT0}), being too cumbersome. 
Some plots of the corresponding PDF are given in Figure 6 for some choices of $\mu$. 
We remark that the corresponding discrete component of such distribution is 
$$
 \mathbb P[X(t)=t\,|\,T_0=\tau]=\frac{{\rm e}^{-\lambda t}\psi_t(\tau-t)}{\psi_0(\tau)},
 \qquad t\in (0,\tau).
$$
\par
Finally, we omit the determination of the conditional distribution of $X(t)$ within 
a renewal cycle in the case of non-zero initial state, since the  involved calculations  
are very cumbersome.
%
%--------------------------------------------------------------------
\begin{figure}[t]  %%%%%%%%%%% FIGURA  6
%\vspace{20mm}
\begin{center}
\centerline{
\epsfxsize=7cm
\epsfbox{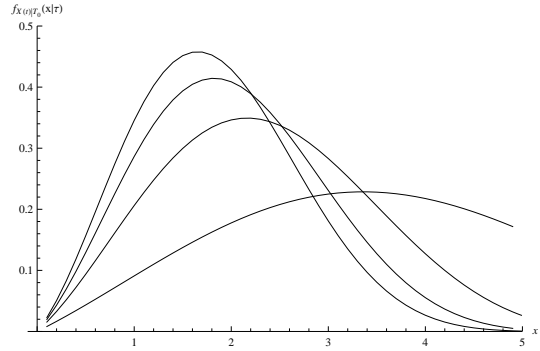}
}
\end{center}
\vspace{-0.5cm}
\caption{Conditional  density  of $X(t)$ given $T_0=\tau$, 
for $\tau=6$, $t=5$ and $\lambda=2$ with $\mu=0.1,0.5,1,1.5$ from bottom to top 
near the origin.}
\end{figure}
%-------------------------------------------------------------------- 

%--------------------------------------------------------------------

%--------------------------------------------------------------------

\end{document}